\documentclass[11pt,a4paper]{amsart}
\usepackage{amsmath,amssymb,amsfonts, float}
\usepackage[all]{xy}
\usepackage{caption}
\usepackage{enumerate}
\usepackage{mathpazo}
\usepackage{a4wide}

\usepackage{bm}
\usepackage{graphicx}
\usepackage[breaklinks=true]{hyperref}

\usepackage{xcolor}
\usepackage{multicol}
\usepackage{tabularx}

\usepackage{hyperref}
\usepackage[capitalize]{cleveref}

\usepackage[normalem]{ulem}

\newtheorem{thm}{Theorem}[section] 
\newtheorem*{thm*}{Theorem} 
\newtheorem{prop}[thm]{Proposition}
\newtheorem{lem}[thm]{Lemma}

\theoremstyle{definition}

\newtheorem{question}[thm]{Question}
\newtheorem{rem}[thm]{Remark}

\DeclareMathOperator{\F}{\mathbb{F}}

\newcommand{\Rad}{{\rm Rad}}
\newcommand{\s}{\text{ss}}

    \DeclareFontFamily{U}{wncy}{}
    \DeclareFontShape{U}{wncy}{m}{n}{<->wncyr10}{}
    \DeclareSymbolFont{mcy}{U}{wncy}{m}{n}
    \DeclareMathSymbol{\Sha}{\mathord}{mcy}{"58}

\numberwithin{equation}{section}

\DeclareSymbolFont{bbold}{U}{bbold}{m}{n}
\DeclareSymbolFontAlphabet{\mathbbold}{bbold}

\begin{document}
\title{A complete classification of \\ perfect unitary Cayley graphs}
 \author{ J\'an Min\'a\v{c}, Tung T. Nguyen, Nguy$\tilde{\text{\^{e}}}$n Duy T\^{a}n }
\address{Department of Mathematics, Western University, London, Ontario, Canada N6A 5B7}
\email{minac@uwo.ca}
\date{\today}

 \address{Department of Kinesiology, Western University, London, Ontario, Canada N6A 5B7}
 \email{tung.nguyen@uwo.ca}
 
  \address{
Faculty of  Mathematics and 	Informatics, Hanoi University of Science and Technology, 1 Dai Co Viet Road, Hanoi, Vietnam } 
\email{tan.nguyenduy@hust.edu.vn}

\thanks{JM is partially supported by the Natural Sciences and Engineering Research Council of Canada (NSERC) grant R0370A01. He gratefully acknowledges the Western University Faculty of Science Distinguished Professorship 2020-2021. TTN is partially supported by an AMS-Simons Travel Grant. NDT is partially supported by the Vietnam National
Foundation for Science and Technology Development (NAFOSTED) under grant number 101.04-2023.21}
\maketitle

\begin{abstract}
Due to their elegant and simple nature, unitary Cayley graphs have been an active research topic in the literature. These graphs are naturally connected to several branches of mathematics, including number theory, finite algebra, representation theory, and graph theory. In this article, we study the perfectness property of these graphs. More precisely, we provide a complete classification of perfect unitary Cayley graphs associated with finite rings.

\end{abstract}
\maketitle

\section{Introduction}
Let $R$ be a finite unital associative ring. The unitary Cayley graph on $R$ is defined as the Cayley graph $G_R = \mathrm{Cay}(R, R^{\times})$, where $R^{\times}$ denotes the group of all invertible elements in $R$. Specifically, $G_R$ is a graph characterized by the following:
\begin{enumerate}
\item The vertex set of $G_R$ is $R$,
\item Two vertices $a, b \in V(G_R)$ are adjacent if and only if $a - b \in R^{\times}$.
\end{enumerate}
Various cases of unitary Cayley graphs have been investigated in the literature. To the best of our knowledge, \cite{klotz2007some} is perhaps the first work that formally introduces the concept of a unitary Cayley graph. In this work, the authors discover some fundamental properties of the unitary graph when $R$ is the ring of integers modulo a given positive integer $n$, such as their spectra, clique, and independence numbers, planarity, perfectness, and much more. \cite{unitary} generalizes many results from \cite{klotz2007some} to the case where $R$ is an arbitrary finite commutative ring. In particular, they are able to classify all perfect unitary Cayley graphs when $R$ is commutative (see \cite[Theorem 9.5]{unitary}). The work \cite{kiani2012unitary} extends this line of research to rings that are not necessarily commutative.

Our interest in unitary Cayley graphs stems from the joint work \cite{chudnovsky2024prime}, where we classify all prime unitary Cayley graphs amongst other things (see \cite[Theorem 4.34]{chudnovsky2024prime}). During our discussions, Sophie Spirkl posed the following question:

\begin{question} \label{question:main}
    Can we classify all perfect unitary Cayley graphs? 
\end{question}
We recall that a graph $G$ is said to be perfect if, for every induced subgraph $H$ of $G$, the chromatic number of $H$ equals the size of its maximum clique. The strong perfect graph theorem gives concrete criteria for a graph to be perfect; namely, a graph is perfect if and only if neither the graph itself nor its graph complement contains an induced odd cycle of length of at least five (see \cite{chudnovsky2006strong}).

The goal of this article is to provide a complete answer to \cref{question:main}. For the precise statement, we refer readers to \cref{thm:main}. We remark that our approach combines experimental and theoretical mathematics. Specifically, we use SageMath to generate unitary Cayley graphs and then utilize algorithms from the Python library NetworkX to identify relevant cycles within these graphs and check whether two given graphs are isomorphic. As it turns out, when they exist these cycles exhibit several common features, enabling us to find a pattern that is applicable to general cases (of course, we need to show that the pattern that we found \textit{actually} works! For that, we need pure mathematics.)

\subsection{Code}
The code that we developed to generate unitary graphs over a finite ring and do experiments on them can be found at \cite{Nguyen_unitary_graph}. 

\section{Reduction to the semisimple case}  \label{sec:reduction}
In this section, we use some structure theorem for finite rings to reduce \cref{question:main} to the case where $R$ is a semisimple ring. We first start with the following observation. While we believe that this observation is well-known among the experts, we cannot find a reference for it. For the sake of completeness, we provide our proof here. 

\begin{prop} \label{prop:left-right}
Let $R$ be a finite ring and $r \in R$. Then the following conditions are equivalent. 
\begin{enumerate}
    \item $r$ is invertible.
    \item $r$ is left invertible. 
    \item $r$ is right invertible. 
\end{enumerate}
\begin{proof}
By definition $(1)$ implies $(2)$ and $(3)$. Additionally $(2)+(3)$ implies $(1)$.  Therefore, it is sufficient to show that $(2)$ and $(3)$ are equivalent. We will show that $(2)$ implies $(3).$ The statement that $(3)$ implies $(2)$ can be proved using the same argument. 

Now, suppose that $(2)$ holds. Let $r \in R$ such that $r$ is left-invertible; i.e., there exists $s \in R$ such that $sr=1.$ Let us consider the multiplication by $r$ map $m_r\colon R \to R$ defined by $m_r(a)=ra.$ This is a an injective group homomorphism on $(R,+)$.  Indeed, let $a \in \ker(m_r)$. Then $0=s(ra)=(sr)a = a$ and hence $\ker(m_r)= 0$. Because $R$ is a finite set, we conclude that $m_r$ is also subjective. In particular, we can find $s'$ such that $rs'=1.$ This shows that $r$ is right invertible and therefore invertible. 
\end{proof}
\end{prop}

\begin{rem}
By \cref{prop:left-right}, there will be no ambiguity when we write $R^{\times}.$
\end{rem}
We recall that the Jacobson radical $\Rad(R)$ of $R$ is the intersection of all left maximal ideals in $R$ (see \cite[Chapter 4.3]{pierce1982associative}). It turns out that $\Rad(R)$ is a two-sided ideal in $R.$ By \cite[Proposition 4.30]{chudnovsky2024prime}, we know that $\Rad(R)$ is a homogenous set in the unitary Cayley graph $G_R.$ Additionally, by for \cite[Corollary 4.2]{chudnovsky2024prime}, we know that 
\[ G_R \cong G_{R^{\s}} * E_n, \]
here $R^{\s} = R/\Rad(R)$ is the simplification of $R$, $E_n$ is the complete graph on $n = |\Rad(R)|$ vertices, and $*$ denotes the wreath product of two graphs (see \cite[Definition 2.5]{chudnovsky2024prime} for the definition of the wreath product of graphs).

\begin{rem}
    Technically speaking, \cite{chudnovsky2024prime} only deals with commutative rings. However, the arguments for \cite[Theorem 4.30, Corollary 4.2]{chudnovsky2024prime} can be applied directly to all finite rings. 
 \end{rem}

By the strong graph theorem, we can see that $G_1 * G_2$ is perfect if and only if each $G_1$ and $G_2$ is. Since the empty graph $E_n$ is perfect, we have the following immediate consequence. 
\begin{prop}
    $G_{R}$ is perfect if and only if $G_{R^{\s}}$ is perfect. 
\end{prop}

Therefore, from now on, we can assume that $R=R^{\s}$; i.e., $R$ is semisimple. By the Artin-Wedderburn structure theorem, we know that 
\[ R = \prod_{i=1}^s R_i \times \prod_{i=1}^r M_{d_i}(F_i).\]
Here $R_i$ is a finite field such that $2 \leq |R_1| \leq |R_2|\leq \cdots \leq |R_s|$. Additionally, $d_i \geq 2$, and $F_i$ is a finite field. We can then see that $G_R$ is a direct product of the unitary 
\[ G_R = \prod_{i=1}^s G_{R_i} \times \prod_{i=1}^r G_{M_{d_i}(\F_i)} = \prod_{i=1}^s K_{|R_i|} \times \prod_{i=1}^r G_{M_{d_i}(\F_i)}.\]

The case where $R$ is commutative is treated in \cite{unitary}. To break down the problem, we will deal with one factor at a time. For these reasons, we will first consider the case where $R$ is a matrix ring over a finite field in the next section. 

\section{$R$ is a matrix ring.}
Let $d \geq 2$ and $F$ a finite field. Let $M_d(F)$ be the ring of $d \times d$ matrices with coefficients in $F.$ Let $G_{M_d(F)}$ the unitary Cayley graph on $M_d(F)$.  We study the following question.
\begin{question}
When is $G_{M_d(F)}$ a perfect graph?  
\end{question}
We remark that there is a diagonal embedding $F^d \to M_d(F).$ Consequently, $G_{F^d}$ is naturally an induced subgraph of $G_{M_d(F)}.$ Consequently, if $G_{M_d(F)}$ is perfect, then so is $G_{F^d}.$ By \cite[Theorem 9.5]{unitary}, we conclude that either $F=\F_2$ or $d=2.$ In summary, we have just proved the following.
\begin{prop}
    If $M_d(F)$ is perfect, then either $d=2$ or $F=\F_2.$
\end{prop}

We remark that when $F=\F_2$, the graph $G_{F^d}$ is bipartite and hence perfect. One can naturally ask whether $G_{M_d(F)}$ is bipartite. The answer is no. To show this, we need to recall the following lemmas. 

\begin{lem}[{See \cite[Proposition 2.6]{bipartite_cayley}}] \label{lem:bipartite}
    Let $G$ be a finite group and $S$ a symmetric subset of $G$. Suppose further that ${\rm Cay}(G,S)$ is connected. Then ${\rm Cay}(G,S)$ is bipartite if and only if there exists an index $2$ subgroup $H$ of $G$ such that $H \cap S = \emptyset.$ 
\end{lem}

\begin{lem}[{See \cite{invertible_matrices}}] \label{lem:sum_of_two}
If $d\geq 2 $ then every matrix in $M_d(F)$ can be written as the sum of two invertible matrices. 
    
\end{lem}

\begin{prop}
If $d \geq 2$ then $G_{M_d(F)}$ is not bipartite.     
\end{prop}
\begin{proof}
    Assume that $G_{M_d(F)}$ is bipartite. Then by \cref{lem:bipartite} we can find an additive subgroup $H$ of $M_d(F)$ with index $2$ such that $H \cap GL_d(F)= \emptyset.$ By this assumption, we conclude that if $a, b \in GL_d(F)$ then $a+b \in H.$ By \cref{lem:sum_of_two}, we conclude that $H=M_d(F)$, which is a contradition. We conclude that $G_{M_d(F)}$ is not bipartite. 
\end{proof}

Using some Sagemath code, we can check that there are no induced odd cycles of length at least $5$ in $G_{M_2(\F_2)}$ (see \cite{Nguyen_unitary_graph} where we do a brute force search for odd-cycles on $G_{M_2(\F_2)}$). However, $G_{M_2(\F_3)}$ contains the following induced $5$-cycle $A_1 \to A_2 \to A_3 \to A_4 \to A_5$, where 

\[ A_1 = \begin{bmatrix} 0 & 0 \\ 0 & 0 \end{bmatrix}, A_2 = \begin{bmatrix} 2 & 1 \\ 1 & 0 \end{bmatrix}, A_3 = \begin{bmatrix} 2 & 2 \\ 2 & 2 \end{bmatrix}, A_4 = \begin{bmatrix} 2 & 1 \\ 1 & 2 \end{bmatrix}, A_5 = \begin{bmatrix} 1 & 2 \\ 2 & 2 \end{bmatrix}.\]
We conclude that $G_{M_2(\F_3)}$ is not perfect. More generally, if ${\rm char}(F) \neq 2$, we can find the following induced $5$-cycle in $G_{M_2(F)}$
\[ A_1 = \begin{bmatrix} 0 & 0 \\ 0 & 0 \end{bmatrix}, A_2 = \begin{bmatrix} 2 & 1 \\ 1 & 0 \end{bmatrix}, A_3 = \begin{bmatrix} -2 & -2 \\ -1 & -1 \end{bmatrix}, A_4 = \begin{bmatrix} -1 & -2 \\ -1  & -2 \end{bmatrix}, A_5 = \begin{bmatrix} -2 & -1 \\ -1 & -1 \end{bmatrix}.\]
We conclude the following proposition. 
\begin{prop}
    If ${\rm char}(F) \neq 2$ then $G_{M_2(F)}$ contains an induced 5-cycle and hence $G_{M_2(F)}$ is not perfect. 
\end{prop}



In the case ${\rm char}(F)=2$ and $F \neq \F_2$, we find the following induced $5$-cycle in $G_{M_2(F)}$. 

\[ A_1 = \begin{bmatrix} 0 & 0 \\ 0 & 0 \end{bmatrix}, A_2 = \begin{bmatrix} 1 & 0 \\ 1 & 1 \end{bmatrix}, A_3 = \begin{bmatrix} z+1 & 0 \\ 1 & 0 \end{bmatrix}, A_4 = \begin{bmatrix} 1 & z+1 \\ 1  & z+1 \end{bmatrix}, A_5 = \begin{bmatrix} z+1 & 1 \\ 1 & 1 \end{bmatrix}.\]
Here $z$ is any element in $F\setminus \F_2$. We conclude that 
\begin{prop}
     If ${\rm char}(F)=2$ and $F \neq \F_2 $ then $G_{M_2(F)}$ contains an induced 5-cycle and hence $G_{M_2(F)}$ is not perfect. 
\end{prop}

For the case $M_d(\F_2)$ with $d \geq 3$, we also have 
\begin{prop}
    If $d \geq 3$, then then $G_{M_d(F)}$ contains an induced 5-cycle and hence $G_{M_d(\F_2)}$ is not perfect. 
\end{prop}
\begin{proof}
We will show that the following vertices produce an induced $5$-cycle (with some help from computers for small $d$) 
\[ 0 \to I+A \to A \to I+A+B^T \to I+B .\]

Here 
\[
A = \begin{bmatrix}
    1 & 1 & \cdots &  0 & 0 &0\\
    1 & 1 & \cdots & 0 & 0 &0\\
    \vdots & \vdots & \ddots & \vdots&\vdots&\vdots \\
    0 & 0 & \cdots &0 &0&0\\
     0 & 0 & \cdots &0 &0&0
\end{bmatrix},\;
\text{ and }
B = \begin{bmatrix}
    0 & 1 & 0 & \cdots &  0 \\
    0 & 0 & 1 & \cdots  & 0 \\
    \vdots & \vdots & \ddots & \vdots \\
    0 & 0 & 0 &\cdots & 1 \\
    0 & 0 & 0 & \cdots  &0
\end{bmatrix}.
\]
In other words, the entries in the upper left $2 \times 2$-minor of $A$ are $1$ while other entries are $0$ and $B$ is a Jordan block of size $d.$

It is straightforward to see that $A^d=B^d=0$. Hence  $I+A$, $I+B$, $(I+A+B^T)-A=I+B^T$ are invertible. We have
$|(I+A)-(I+B)|=|A-B|=0$ since the $d$th row of $A-B$ is the zero row. 
We also have $|I+A+B^T|=0$ since the first column of $I+A+B^T$ is the zero column. On the other hand, $|(I+B)-A|=0$ since the first row of $I+B-A$ is the zero row.

Now we only need to show that $(I+A+B^T)-(I+B)=A+B+B^T$ is invertible. For convenience, for each $n\geq 2$ let $D_n=B_n+B_n^T$, where $B_n$ is a Jordan block of size of $n$. By expanding the determinant of $D_n$ along the first column and then along the first row we get $|D_n|=|D_{n-1}|$. Hence $|D_n|=|D_2|=1$ if $n$ is even and $|D_n|=|D_3|=0$ if $n$ is odd.

Now 
\[
\begin{aligned}
|A+B+B^T|&=\begin{vmatrix}
    1 & 0 & 0 & 0&\cdots &  0 &0\\
    0 & 1 & 1 & 0& \cdots  & 0 &0\\
    0 & 1 & 0 & 1& \cdots  & 0 &0\\
     0 & 0 & 1 & 0& \cdots  & 0 &0\\
    \vdots & \vdots & \ddots & \vdots \\
    0 & 0 & 0 & 0 &\cdots & 0 &1 \\
    0 & 0 & 0 & 0& \cdots &1 &0
\end{vmatrix}\\&=
\begin{vmatrix}
     1 & 1 & 0& \cdots  & 0 &0\\
   1 & 0 & 1& \cdots  & 0 &0\\
      0 & 1 & 0& \cdots  & 0 &0\\
    \vdots & \vdots & \ddots & \vdots \\
    0 & 0 & 0 &\cdots & 0 &1 \\
    0 & 0 & 0& \cdots &1 &0
\end{vmatrix} (\text{expanding the determinant along the 1st row})\\
&=|D_{d-2}|- \begin{vmatrix}
     1 & 1 & 0& \cdots  & 0 &0\\
   0 & 0 & 1& \cdots  & 0 &0\\
      0 & 1 & 0& \cdots  & 0 &0\\
    \vdots & \vdots & \ddots & \vdots \\
    0 & 0 & 0 &\cdots & 0 &1 \\
    0 & 0 & 0& \cdots &1 &0
\end{vmatrix} (\text{expanding the determinant along the 1st row})\\
&=|D_{d-2}|-|D_{d-3}|  (\text{expanding the second determinant along the 1st column})\\
&=1.
\end{aligned}
\]
Hence $A+B+B^T$ is invertible and we are done.
\end{proof}
In summary, we have proved the following theorem. 
\begin{thm}
$G_{M_d(F)}$ is perfect if and only $d=2$ and $F = \F_2. $
\end{thm}
\section{The general case}

As discussed at the end of\cref{sec:reduction}, for \cref{question:main} we can assume that $R$ is semi-simple. Furthermore, we can suppose that $R$ has the following structure 
\[ R = \prod_{i=1}^s R_i \times \prod_{i=1}^r M_{d_i}(F_i).\]
Here $R_i$ is a finite field such that $0 \leq |R_1| \leq |R_2|\leq \cdots \leq |R_s|$. Additionally, $d_i \geq 2$, and $F_i$ is a finite field. With this decomposition, $G_R$ is the the following direct product
\[ G_R = \prod_{i=1}^s G_{R_i} \times \prod_{i=1}^r G_{M_{d_i}(\F_i)} = \prod_{i=1}^s K_{|R_i|} \times \prod_{i=1}^r G_{M_{d_i}(\F_i)}.\]
If $|R_1|=2$ then $G_{R}$ is bipartite, hence perfect. Therefore, we can assume that $|R_i| \geq 3$ for all $1 \leq i \leq s.$ 
\begin{lem} Suppose $R_1,R_2,R_3$ are finite fields such that $|R_i|\geq 3$, for all $1\leq i\leq 3$. Then $G_{R_1}\times G_{R_2}\times G_{R_3}$ contains an induced 5-cycle.
\end{lem} 
\begin{proof} Let $a\in R_2\setminus\{0,1\}$, $b\in R_3\setminus\{0,1\}$ and $c\in R_1\setminus\{0,1\}$. Then one can check that
\[
(0,0,0)\to (1,1,1)\to (0,a,b)\to (1,1,0)\to (c,a,1)
\]
is an induced 5-cycle.
\end{proof}

\begin{lem}
    Let $G_1, G_2, \ldots, G_d$ be graphs and $k \geq 2$ a fixed integer. Suppose that for each $1 \leq i \leq d$, 
    \[ v_{i1} \to v_{i2} \to v_{i3} \to v_{i4}\to\cdots \to v_{ik} \] 
    is a closed path of length $k$ in $G_i$ (we do not require that $\{v_{ij}\}_{1 \leq j \leq k}$ are different, except for $i=1$). Suppose further that the induced graph on $\{v_{11}, v_{12}, v_{13}, \ldots, v_{1k} \}$ is a $k$-cycle. For each $1 \leq j \leq k$, let
    \[ v_{j} = (v_{ij})_{1 \leq i \leq d} \in G_1 \times G_2\times \cdots \times G_d.\]
    Then the induced graph on $\{v_{j})_{1 \leq j \leq k}$ is an induced $k$-cycle. 
\end{lem}
\begin{proof}
    Clearly $v_1,v_2,\ldots,v_k$ are pairwise different and $v_1\to v_2\to\cdots\to v_k$ is a closed path. Since the induced graph on $\{v_{11}, v_{12}, v_{13}, \ldots, v_{1k} \}$ is a $k$-cycle, we see that if $i\not=j$ then $v_{1i}$ and $v_{1j}$ are not connect, and hence $v_i$ and $v_j$ are not connected as well.
\end{proof}

By our assumption that $3 \leq |R_1| \leq |R_2|\leq  \cdots \leq |R_s|$, $K_3$ is always a subgraph of $G_{R_i}$. By Lemma \ref{lem:sum_of_two}, we also know that $K_3$ is a subgraph of $G_{M_{d_i}(F_i)}$ as well. Since $K_3$ contains a closed path of length $5$ for each $k \geq 1$, $G_{R_i}$ and $G_{M_{d_i}(F_i)}$ both contain a closed path of length $5$. We conclude that if $R$ contains a factor of the form $M_d(F) \neq M_2(\F_2)$ or $s \geq 3$ then $G_R$ is perfect as it will contain an induced $5$-cycle. We now deal with the remaining cases, namely 
\[ R = \prod_{i=1}^s R_i \times (M_2(\F_2))^r,\]
where $s \leq 2$ and $r \geq 0.$ For these cases, we have the following observation. 

\begin{lem}
    Suppose that $R = M_2(\F_2) \times F$ where $F$ is a field such that $F \neq \F_2.$ Let $\alpha \in F \setminus \F_2.$ Then the following elements produce an induced $5$-cycle on $G_R$ 

    \[ r_1 = \left(\begin{bmatrix} 0 & 0 \\ 0 & 0 \end{bmatrix}, 0 \right), r_2 = \left(\begin{bmatrix} 1 & 1 \\ 0 & 1 \end{bmatrix}, 1 \right), r_3 = \left(\begin{bmatrix} 0 & 0 \\ 1 & 1 \end{bmatrix}, \alpha \right) \\ r_4 = \left(\begin{bmatrix} 1 & 0 \\ 0  & 0 \end{bmatrix}, 0 \right), r_5 =  \left(\begin{bmatrix} 1 & 1 \\ 1 & 0 \end{bmatrix}, \alpha \right).\]
Consequently, $G_R$ is not perfect. 
\end{lem}
For the case $R = M_2(F_2) \times M_2(\F_2)$ we have the following 
\begin{lem}
    Suppose that $R = M_2(\F_2) \times M_2(\F_2)$ Then the following elements produce an induced $5$-cycle on $G_R$ 
    \[ r_1 = \left(\begin{bmatrix} 0 & 0 \\ 0 & 0 \end{bmatrix}, \begin{bmatrix} 0 & 1 \\ 1 & 1 \end{bmatrix}  \right), r_2 = \left(\begin{bmatrix} 1 & 0 \\ 1 & 1 \end{bmatrix}, \begin{bmatrix} 1 & 0 \\ 0 & 1 \end{bmatrix}  \right), r_3 = \left(\begin{bmatrix} 1 & 1 \\ 0 & 1 \end{bmatrix}, \begin{bmatrix} 1 & 1 \\ 1 & 1 \end{bmatrix}  \right) ,\] 
    \[ r_4 = \left(\begin{bmatrix} 1 & 0 \\ 1 & 1 \end{bmatrix}, \begin{bmatrix} 1 & 0 \\ 0 & 0 \end{bmatrix}  \right), r_5 =  \left(\begin{bmatrix} 1 & 1 \\ 0 & 1 \end{bmatrix}, \begin{bmatrix} 0 & 0 \\ 0 & 1 \end{bmatrix}  \right).\]
Consequently, $G_R$ is not perfect. 
\end{lem}

In summary, we have the following theorem which classifies all perfect unitary Cayley graphs. 

\begin{thm} \label{thm:main}
    Suppose that $R$ is a finite ring such that its semisimplification has the following decomposition 
    \[ R^{\s} = \prod_{i=1}^s R_i \times \prod_{i=1}^r M_{d_i}(F_i).\]
Here $R_i$ is a finite field such that $2 \leq |R_1| \leq |R_2|\leq  \ldots \leq |R_s|$. Additionally, $d_i \geq 2$, and $F_i$ is a finite field. Then $G_R$ is a perfect graph if and only if one of the following conditions is satisfied. 
\begin{enumerate}
    \item $|R_1|=2$. In this case, $G_R$ is a bipartite graph and hence perfect. 
    \item $s \leq 2$ and $r=0.$ In other words, either $R^{\s} = R_1$ or $R^{\s} = R_1 \times R_2.$
    \item $R^{\s} = M_2(\F_2).$
\end{enumerate}
\end{thm}

\section*{Acknowledgements}
We thank Maria Chudnovsky, Michal Cizek, Logan Crew, Sophie Spirkl, and Torsten Sander for the enlightening discussions about unitary graphs and related topics. We are especially grateful to Sophie Spirkl for raising \cref{question:main}. We also thank Sunil Chebolu for introducing us to the article \cite{invertible_matrices}.

\bibliographystyle{amsplain}
\bibliography{references.bib}

\providecommand{\bysame}{\leavevmode\hbox to3em{\hrulefill}\thinspace}
\providecommand{\MR}{\relax\ifhmode\unskip\space\fi MR }
\providecommand{\MRhref}[2]{%
  \href{http://www.ams.org/mathscinet-getitem?mr=#1}{#2}
}
\providecommand{\href}[2]{#2}
\begin{thebibliography}{1}

\bibitem{unitary}
Reza Akhtar, Megan Boggess, Tiffany Jackson-Henderson, Isidora Jim{\'e}nez, Rachel Karpman, Amanda Kinzel, and Dan Pritikin, \emph{On the unitary {Cayley} graph of a finite ring}, Electron. J. Combin. \textbf{16} (2009), no.~1, Research Paper 117, 13 pages.

\bibitem{bipartite_cayley}
Arindam Biswas, \emph{On a cheeger type inequality in {Cayley} graphs of finite groups}, European Journal of Combinatorics \textbf{81} (2019), 298--308.

\bibitem{chudnovsky2024prime}
Maria Chudnovsky, Michal Cizek, Logan Crew, J{\'a}n Min{\'a}{\v{c}}, Tung~T Nguyen, Sophie Spirkl, and Nguy{\^e}n~Duy T{\^a}n, \emph{On prime {Cayley} graphs}, arXiv preprint arXiv:2401.06062 (2024).

\bibitem{chudnovsky2006strong}
Maria Chudnovsky, Neil Robertson, Paul Seymour, and Robin Thomas, \emph{The strong perfect graph theorem}, Annals of mathematics (2006), 51--229.

\bibitem{kiani2012unitary}
Dariush Kiani and Mohsen Molla~Haji Aghaei, \emph{On the unitary {Cayley} graph of a ring}, Electron. J. Combin. (2012), P10--P10.

\bibitem{klotz2007some}
Walter Klotz and Torsten Sander, \emph{Some properties of unitary {Cayley} graphs}, Electron. J. Combin. (2007), R45--R45.

\bibitem{invertible_matrices}
N.~J. Lord, \emph{Matrices as sums of invertible matrices}, Mathematics Magazine (1987).

\bibitem{Nguyen_unitary_graph}
J{\'a}n Minac, Tung~T. Nguyen, and Nguyen~Duy T{a}n, \emph{Unitary cayley graphs over a finite ring}, \url{https://github.com/tungprime/perfect_unitary_cayley}, 2024.

\bibitem{pierce1982associative}
Richard~S. Pierce, \emph{Associative algebras}, Studies in the History of Modern Science, vol.~9, Springer-Verlag, New York-Berlin, 1982. \MR{674652}

\end{thebibliography}
\end{document}